\documentclass[11pt]{amsart}
\usepackage{amsfonts,amscd}
\usepackage{amssymb}
\usepackage{url}
\usepackage{epsfig}

\usepackage{pstricks}
\usepackage{pst-plot}
\usepackage{pst-node}
\usepackage[fleqn,tbtags]{mathtools}

\def \r{\mathbb R}
\def \s{\mbox{${\mathbb S}$}}

\def \h{\mathbb  H}

\DeclareMathOperator{\sech}{sech}
\DeclareMathOperator{\grad}{grad}

\DeclareMathOperator{\cst}{constant}
\DeclareMathOperator{\arccot}{arccot}

\theoremstyle{plain}
\newtheorem{theorem}                 {Theorem}      [section]
\newtheorem{proposition}  [theorem]  {Proposition}

\newtheorem{lemma}        [theorem]  {Lemma}

\theoremstyle{definition}
\newtheorem{example}      [theorem]  {Example}
\newtheorem{remark}       [theorem]  {Remark}

\begin{document}


\title{Geodesics on an invariant surface}

\author{Stefano Montaldo}
\address{Universit\`a degli Studi di Cagliari\\
Dipartimento di Matematica e Informatica\\
Via Ospedale 72\\
09124 Cagliari, Italia}

\email{montaldo@unica.it}

\author{Irene I. Onnis}

\address{Departamento de Matem\'{a}tica, C.P. 668\\ ICMC,
USP, 13560-970, S\~{a}o Carlos, SP\\ Brasil}

\email{onnis@icmc.usp.br}

\thanks{The first author was supported by: Fondo per il sostegno della ricerca di base per lo start-up dei giovani ricercatori, University of Cagliari - Italy.
The second author was supported by: Visiting professor program, Regione Autonoma della Sardegna - Italy}

\begin{abstract}
We study the geodesics on an invariant surface of a three
dimensional Riemannian manifold. The main results are: the
characterization of geodesic orbits; a Clairaut's relation and its
geometric interpretation in some remarkable three dimensional spaces; 
the local description of the geodesics; the explicit
description of geodesic curves on an invariant surface with
constant Gauss curvature.
\end{abstract}

\maketitle

\section{Introduction and Preliminaries}
The theory of surfaces in  three dimensional manifolds is having, in the last decades, a new golden age
evidenced by the great number of papers on the subject. 
%(see, for examples, \cite{ABRO,DA,FEMI} and, in particular, the works of Harold Rosenberg and his collaborators on the subject \cite{Rosenberg}). 
An important geometric class of surfaces in a three dimensional manifold is that of invariant surfaces, that is, as described below, surfaces which are invariant under the action of a one-parameter group of isometries of the ambient space. Invariant surfaces have been classified,
according to the value of their Gaussian or mean curvature, in many remarkable three dimensional spaces (see, for example, \cite{RCPPAR1,RCPPAR2,FIMEPE,LO,MO1,MO,Onnis,SETO,SE,TO}). 

In this paper we consider the problem of understanding the geodesics
on an invariant surfaces of a three dimensional manifold.  

To this aim we  briefly recall the definition and the geometry of invariant surfaces.

Let
$({N}^3,g)$ be a three dimensional Riemannian manifold and let $X$
be a Killing vector field on ${N}$. Then $X$ generates a
one-parameter subgroup $G_X$ of the group of isometries of
$({N}^3,g)$. Let now $f:{M}^2\to ({N}^3,g)$ be an immersion from a
surface ${M}^2$ into ${N}^3$ and assume that $f({M})\subset {N}_r$
(the regular part of $N$,  that is, the subset consisting of
points belonging to principal orbits). We say that $f$ is a $G_X$-{\it
equivariant} immersion, and $f({M})$ a $G_X$-{\it invariant} surface
of ${N}$, if there exists an action of $G_X$ on ${M}^2$ such that
for any $x\in {M}^2$ and $a\in G_X$ we have $f(a\,x)=a f(x)$.

A $G_X$-equivariant immersion $f:{M}^2\to ({N}^3,g)$ induces on ${M}^2$ a Rieman\-nian metric,
the pull-back metric, denoted by $g_f$ and called the $G_X$-{\it invariant induced metric}.

Let $f:{M}^2\to ({N}^3,g)$ be a $G_X$-equivariant immersion and let us endow
${M}^2$ with the $G_X$-invariant induced metric $g_f$. 
Assume that $f({M}^2)\subset
{N}_r$ and that ${N}/G_X$ is connected. Then $f$ induces an
immersion $\tilde{f}:{M}/G_X\rightarrow {N}_r/G_X$ between the orbit
spaces and, also, the space ${N}_r/G_X$ can be equipped with a
Riemannian metric, the {\it quotient metric}, so that the quotient
map $\pi:{N}_r\to {N}_r/G_X$ becomes  a Riemannian submersion.

For later use we describe the quotient metric of the regular part of the orbit space
$N/G_X$. It is well known (see, for example, \cite{Olver})  that
$N_r/G_X$ can be locally parametrized by the invariant functions of
the Killing vector field $X$.  If $\{\xi_1,\xi_2\}$ is a complete set of
invariant functions on a $G_X$-invariant subset  of $N_r$, then the
quotient metric is given by  $\tilde{g}=~\sum_{i,j=1}^{2} h^{ij}
d\xi_i\otimes d\xi_j$ where $(h^{ij})$ is the inverse of the matrix
$(h_{ij})$ with entries $h_{ij}=g(\nabla \xi_i,\nabla \xi_j)$. 

We can picture the above construction using the following diagram:
$$
\begin{CD}
({M}^2,g_f) @>f>> ({N_r}^3,g)\\
@V VV @V\pi VV\\
{M}^2/G_X @>\tilde{f}>> ({N}^3_{r}/G_X,\tilde{g}).
\end{CD}
$$

Using the above setting we can give a local description of the $G_X$-invariant surfaces of
${N}^3$. Let $\tilde{\gamma}:(a,b)\subset\r\to({N}^3/G_X,\tilde{g})$
be a curve parametrized by arc length and let
$\gamma:(a,b)\subset\r\to {N}^3$ be a lift of $\tilde{\gamma}$, such
that $d\pi(\gamma')=\tilde{\gamma}'$. If we denote by
$\phi_v,\;v\in(-\epsilon,\epsilon)$, the local flow of the Killing
vector field $X$, then the map
\begin{equation}\label{eq-psi}
\psi:(a,b)\times(-\epsilon,\epsilon)\to {N}^3\,,\quad \psi(u,v)=\phi_v(\gamma(u)),
\end{equation}
defines a parametrized
$G_X$-invariant surface.

Conversely, if $f({M}^2)$ is a $G_X$-invariant immersed surface in ${N}^3$,
then $\tilde{f}$ defines a curve in $({N}^3/G_X,\tilde{g})$ that can be locally
parametrized by arc length. The curve $\tilde{\gamma}$ is generally called the {\it profile}
curve of the invariant surface.

Observe that, as the $v$-coordinate curves are the orbits of the
action of the one-parameter group of isometries $G_X$, the
coefficients of the pull-back metric $g_f=E\, du^2 +2 F\, dudv + G\,
dv^2$ are function only of $u$ and are given by:
$$
\begin{cases}
E=g(\psi_u,\psi_u)=g(d\phi_v(\gamma'),d\phi_v(\gamma'))\\
F=g(\psi_u,\psi_v)=g(d\phi_v(\gamma'),X)\\
G=g(\psi_v,\psi_v)=g(X,X).\\
\end{cases}
$$
Putting  $\omega^2(u):=\|X(\gamma(u))\|^2_g=G$, we have that  (see
\cite{MO})
\begin{equation}
\label{eq-E} E\,G-F^2=G=\omega(u)^2.
\end{equation}
\begin{remark}
Note that \eqref{eq-E} is immediate in the case $\gamma$ is a
horizontal lift of $\tilde{\gamma}$. In fact, in this case, $F=0$
and $E=1$. This fact might suggest to consider always the case when $\gamma$ is a horizontal lift. However, in many cases (see Remark~\ref{re:lifthor}), it could be rather difficult to find a horizontal lift.
Thus it is more convenient to write down the theory in the general case without the assumption that $E=1$ and $F=0$. We will see that everything works nicely thanks to \eqref{eq-E}.
\end{remark}
 Using \eqref{eq-E} and Bianchi's formula for the Gauss curvature
we find that
\begin{equation}\label{eq-K}
K(u)=-\frac{\omega_{uu}(u)}{ \omega(u)}.
\end{equation}
As an immediate consequence we have
\begin{theorem}[\cite{MO}]\label{teo-main}
Let $f:M^2\to(N^3,g)$ be a $G_X$-equivariant immersion,
$\tilde{\gamma}:(a,b)\subset\r\to(N_r^3/G_X,\tilde{g})$
 a parametrization by arc length of the profile curve of $M$ and $\gamma$ a lift of $\tilde{\gamma}$.
Then, the induced metric $g_f$ is of constant Gauss curvature $K$ if
and only if the function $\omega(u)$
 satisfies the following differential equation
\begin{equation}\label{eq-main}
\omega_{uu}(u) + K \omega(u) =0.
\end{equation}

\end{theorem}

\section{Geodesic equations and the Clairaut's relation}
Let $M^2\subset(N^3,g)$ be a $G_X$-invariant surface, locally parametrized by
\eqref{eq-psi}, then the induced metric is 
$$
g_f=E(u)\,du^2+2\,F(u)\,du\,dv+\omega^2(u)\,dv^2.
$$ 
Now let $\alpha(s)=\psi(u(s),v(s))$ be a geodesic parametrized by arc length, then 
$u(s)$ and $v(s)$  satisfy the Euler Lagrange system 
\begin{equation}\label{eq:e-l-system}\left\{\begin{aligned}
\frac{d}{ds}\bigg(\frac{\partial L}{\partial {u'}}\bigg)-\frac{\partial L}{\partial u}&=0,\\
\frac{d}{ds}\bigg(\frac{\partial L}{\partial {v'}}\bigg)-\frac{\partial L}{\partial v}&=0,
\end{aligned}\right.
\end{equation}
where $L=1/2 [E(s){u'(s)}^2+2F(s)u'(s)v'(s)+\omega^2(s){v'(s)}^2]$. Note that with $()'$ we have denoted the derivative with respect to $s$ and when we  restrict a function $h$ defined on $M$ to a curve $\alpha(s)$
we have used the notation $h(s)$. Moreover, in the sequel, to simplify the notation, we will omit the explicit dependency on the coordinates, when this does not create confusion.
Expanding \eqref{eq:e-l-system} we have
\begin{equation}\label{eqgeod}
\left\{\begin{aligned}
E\,{u''}+F\,{v''}+\frac{E_u\,{u'}^2}{2}-\omega\,\omega_u\,{v'}^2&=0,\\
(F\,{u'}+\omega^2\,{v'})'&=0,
\end{aligned}\right.
\end{equation}
where we have denoted by $()_u$ the derivative with respect to $u$.

\begin{proposition}\label{orbit}
Let ${M}$ be a $G_X$-invariant surface of $({N}^3,g)$. Then an orbit $\alpha$ is
a geodesic on ${M}$ if and only if ${(\grad_{M}\omega)}_{|\alpha}=0$.
\end{proposition}
\begin{proof} Parametrizing  the surface ${M}$, locally,
by $\psi(u,v)$ (see \eqref{eq-psi}) the parametrization by arc length of an orbit $u=\cst=u_0\in(a,b)$ is given by
$$
\alpha (s)=\psi(u(s),v(s))=\psi\Big(u_0,\frac{s}{\omega(u_0)}\Big).
$$
Then the second equation of \eqref{eqgeod} automatically holds while the first becomes $\omega_u(u_0)=0$. Now, taking into account  \eqref{eq-E}, the gradient of $\omega$ becomes
$$
\grad_{M}\omega=\omega_u(\frac{\partial }{\partial u}-\frac{F}{G}\frac{\partial }{\partial v})
$$
and we conclude.
\end{proof}

\begin{remark}\label{re:eqgeod-bis}
If $\alpha$ is not an orbit (i.e. $u'(s)\neq 0$),
\eqref{eqgeod} is equivalent to
\begin{equation}\label{eqgeod-bis}
\left\{\begin{aligned}
\|\alpha'\|^2=u'^2 E+2 u' v' F + v'^2 \omega^2&=1,\\
(F\,{u'}+\omega^2\,{v'})'&=0.
\end{aligned}\right.
\end{equation}
To see this  we only have to show that \eqref{eqgeod-bis} implies \eqref{eqgeod}.
Differentiating  with respect to $s$ the equation  
 $$
 1=g(\alpha',\alpha')= E\, u'^2+2\, F\,u'\,v'+\omega^2\,v'^2
 $$ 
 and using that $\alpha$ is not an orbit, we find
\begin{equation}\nonumber
E\,u''+F\, v''+\frac{E_u\,u'^2}{2}=-\frac{( F\,v'\,u''+F_u\,u'^2\,v'+\omega^2\,v'\,v'')}{u'}-\omega\,\omega_u\,v'^2.
\end{equation}
The latter gives
\begin{equation}\nonumber
\begin{aligned}
E\,u''+F\, v''+\frac{E_u\,u'^2}{2}-\omega\,\omega_u\,{v'}^2&=-\frac{{v'}}{{u'}}\,(F\,u''+F_u\,u'^2+2\,\omega \,\omega_u\,u'\,v'+\omega^2\, v'')\\
&=-\frac{{v'}}{{u'}}\,(F\,u'+\omega^2\,v')'\\&=0.
\end{aligned}
\end{equation}
\end{remark}

\begin{proposition}
Let $\alpha(s)$ be a geodesic parametrized by arc length on a $G_X$-invariant surface $M\subset({N}^3,g)$ 
which is orthogonal to all the orbits that it meets. Then $\alpha$ is a geodesic.
\end{proposition}
\begin{proof}
We can, locally, parametrize $\alpha$ by $\alpha(s)=\psi(u(s),v(s))$, where $\psi$ is the local parametrization of $M$ given in \eqref{eq-psi}. Since $\alpha$ cannot be an orbit, we only have to show that the second equation of \eqref{eqgeod-bis} is satisfied.  From $g(\alpha',X)=0$ we get
$F u'+\omega^2 v'=0$.
\end{proof}

%\begin{remark}
%If $\gamma$ is a
%horizontal lift of $\tilde{\gamma}$ (i.e. $F=0$
%and $E=1$), system~\eqref{eqgeod} reduces to
%\begin{equation}\label{eqhorizontal}\left\{\begin{aligned}
%{u''}-\omega\,\omega_u\,{v'}^2&=0,\\
%(\omega^2\,{v'})'&=0.
%\end{aligned}\right.
%\end{equation}
%In this case, all $u$-parameter curves are geodesics.
%In fact, let 
%$$
%\alpha (s)=\psi(u(s),v_0),\qquad {u'}\neq 0,
%$$
%be an arc length parametrization of the curve $v=v_0\in (-\epsilon,\epsilon)$. 
%Then the second equation of \eqref{eqhorizontal} automatically holds. Also, differentiating
%$$
%1=||\alpha'(s)||^2={u'}^2(s),
%$$ 
%as ${u'}\neq 0$, we obtain ${u''}=0$; so the first equation of \eqref{eqhorizontal} holds.
%\end{remark}

\begin{theorem}[Clairaut's Theorem]
Let $\alpha(s)$ be a geodesic parametrized by arc length on a $G_X$-invariant surface $M\subset({N}^3,g)$ and let $\theta(s)$ be the angle under which the curve $\alpha$ meets the orbits of $X$. Then
\begin{equation}\label{clairaut}
    \omega(s)\,\cos\theta(s)=c=\cst.
\end{equation}
Conversely, if $\omega\cos\theta$ is constant along an arc length parametrized curve $\alpha$ on $M$, that is not an orbit of $M$, then $\alpha$ is a geodesic.
\end{theorem}
\begin{proof}
Locally the surface $M$ can be parametrized by \eqref{eq-psi}  and the curve $\alpha$ by
$\alpha(s)=\psi(u(s),v(s))$. Since $\alpha$ is a geodesic, from the second equation of \eqref{eqgeod}, we have
\begin{equation}\label{eqcost}
    F(s)\,{u'}(s)+\omega(s)^2\,{v'}(s)=c\in\r.
\end{equation}
Then the angle $\theta(s)$ satisfies 
\begin{eqnarray}
\omega(s) \cos\theta(s)&=&g(\alpha',X)= g(\alpha',\psi_v)\nonumber\\
&=&F(s)\,{u'}(s)+\omega(s)^2\,{v'}(s)=c.
\end{eqnarray}
Conversely, let $\alpha(s)=\psi(u(s),v(s))$ be a curve on $M$ parametrized by arc length such that  $\omega(s)\,\cos\theta(s)=c\in\r$ along $\alpha$.  Assume that $\alpha$ is not an orbit, then, taking into account Remark~\ref{re:eqgeod-bis}, we only have to show that the second equation of \eqref{eqgeod-bis} is satisfied. We have
$$
F(s)\,{u'}(s)+\omega(s)^2\,{v'}(s)=  g(\alpha',\psi_v)=\omega(s)\,\cos\theta(s)=c.
$$
\end{proof}

We call the constant $c$ associated with each geodesic $\alpha$ the {\it slant} of $\alpha$. Note that the geodesics with slant $c=0$ are those orthogonal to the orbits.

\begin{remark}\label{re:omegac}
Since $|\cos\theta(s)|\leq1$,  \eqref{clairaut} implies that $\omega(s)\geq |c|$, hence $\alpha$ must lies entirely in the region of the invariant surface where $\omega \geq |c|$. Moreover, if $\alpha$ is not an orbit then $\omega>|c|$.
\end{remark}

\begin{example}[Rotational surfaces in $\r^3$]
We considerer the case of rotational surfaces in the Euclidean three dimensional space $(\r^3,g)$, with $g=dx^2+dy^2+dz^2$, assuming (without loss of generality) that the rotation is about the $z$-axes. Then the Killing vector field is
$X=y\frac{\partial}{\partial x}-x\frac{\partial}{\partial y}$.
In this case the Clairaut' relation \eqref{clairaut} becomes the classical ones
$$
r(s)\,\cos\theta(s)= c,
$$
where $r(s)$ represents the radius of the orbit.
\end{example}

\subsection{The Clairaut's relation for invariant surfaces in $\h^2\times\r$}
Let $\h^{2}=\{(x,y)\in\r^2\,:\,y>0\}$ be the half plane model of the
hyperbolic plane and consider $\h^2\times\r$ endowed with the
product metric
\begin{equation}\label{eq:metric-h2xr}
g=\frac{dx^2+dy^2}{y^2}+dz^2.
\end{equation}

The Lie algebra of the infinitesimal isometries of the product
$(\h^2\times\r,g)$ admits the following bases of Killing vector
fields
\begin{align*}
X_1&=\frac{(x^2-y^2+1)}{2}\frac{\partial}{\partial
x}+xy\frac{\partial}{\partial
y}\\
X_2&=\frac{\partial}{\partial x}\\
X_3&=x\frac{\partial}{\partial x}+y\frac{\partial}{\partial y}\\
X_4&=\frac{\partial}{\partial z}.
\end{align*}

%Let denote by $G_i$  the one-parameter subgroup of isometries
%generated by $X_i$, by $G_{ij}$  the one-parameter subgroup of
%isometries generated by linear combinations of $X_i$ and $X_j$ and
%so on. 
The class of invariant surfaces in $\h^2\times\r$ can be
divided into three subclasses according to the following

\begin{proposition}[\cite{Onnis}]\label{lem-red}
Any surface in $\h^2\times\r$ which is invariant under the action of
a one-parameter subgroup of isometries $G_X$, generated by a Killing
vector field $X=\sum_{i=1}^4a_i X_i$, $a_i\in\r$, is congruent to a
surface invariant under the action of one of the following groups
$$
G_{14}=G_{X_1+bX_4}, \quad G_{24}=G_{aX_2+bX_4},\quad G_{34}=G_{X_3+bX_4},
$$
where $a,b\in\r.$
\end{proposition}

To understand the shape of an invariant surface in $\h^2\times\r$ we need to describe the
orbits of the three groups  $G_{24}, G_{34}$ and $G_{14}$.
A direct computation shows that the orbit of a point $p_0=(x_o,y_o,z_o)\in \h^2\times\r$ is:

$\bullet$\quad {\em under the action of $G_{24}$} the curve parametrized by
$$
(a\,v+x_o,y_o,b\, v +z_o),\quad v\in(-\epsilon,\epsilon),
$$
which looks like an Euclidean line on the plane $y=y_0$;

$\bullet$\quad {\em under the action of $G_{34}$} the curve parametrized by
\begin{equation}\label{eq:orbitsg34}
(e^v x_o,e^v y_o,b\, v+z_o),\quad v\in(-\epsilon,\epsilon),
\end{equation}
which belongs to a vertical plane through the $z$-axes and looks like a logarithms curve;

$\bullet$\quad {\em under the action of $G_{14}$} the curve parametrized by
$$
(x(v),y(v),b\, v+z_o),\quad v\in(-\epsilon,\epsilon),
$$
where 
$$ 
x(v)^2+y(v)^2-\beta\, y(v)+1=0,\quad \beta=\frac{1+x_0^2+y_0^2}{y_0},
$$
which looks like an Euclidean helix in a right circular cylinder with Euclidean axes in the plane $x=0$.

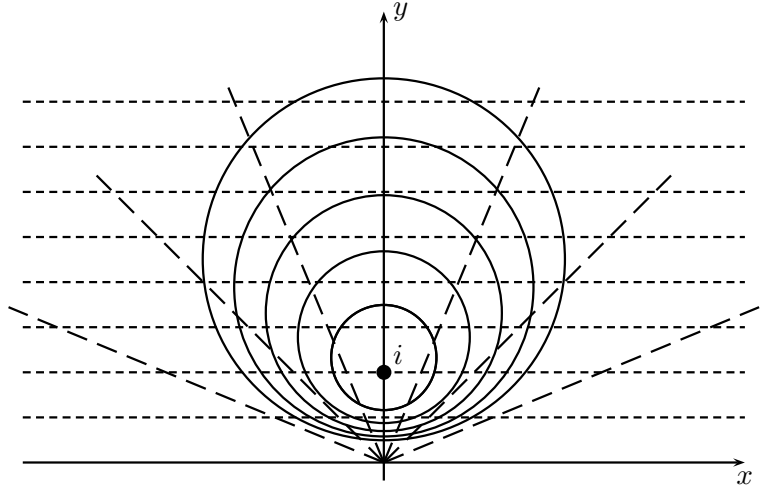
\begin{figure}[h!]
\begin{pspicture}(-4,-0.5)(4,5.9)
\psset{xunit=1.2cm,yunit=1.2cm,linewidth=.03}
\psaxes[labels=none,ticks=none]{->}(0,0)(-4,-0.2)(4,5)
\rput[t](4,-0.1){$x$}
\rput[l](0.1,5){$y$}

\rput[l](0.1,1.2){$i$}
%\rput[b]{45}(3,3){\uput[d](0,0){$G_{34}$}}
%\rput[b]{0}(3,4){\uput[d](0,0){$G_{24}$}}

\psdots[dotsize=0.2](0,1)

%\pscircle[runit=3cm](0,1.01752){0.0940094}
\pscircle[runit=2.4cm](0,1.16388){0.297751}
\pscircle[runit=2.4cm](0,1.3895){0.482366}
\pscircle[runit=2.4cm](0,1.6562){0.660112}
\pscircle[runit=2.4cm](0,1.94693){0.835243}
\pscircle[runit=2.4cm](0,1.16388){0.297751}
\pscircle[runit=2.4cm](0,2.25293){1.00942}
%\pscircle(0,1.2475){0.372909}
%\pscircle(0,1.29289){0.409748}

\psplot[linestyle=dashed,dash=8pt 4pt]{0}{4.15746}{x  0.414214 mul}
\psplot[linestyle=dashed,dash=8pt 4pt]{0}{3.18198}{x 1 mul}
\psplot[linestyle=dashed,dash=8pt 4pt]{0}{1.72208}{x 2.41421 mul}
\psplot[linestyle=dashed,dash=8pt 4pt]{0}{-1.72208}{x -2.41421 mul}
\psplot[linestyle=dashed,dash=8pt 4pt]{0}{-3.18198}{x -1 mul}
\psplot[linestyle=dashed,dash=8pt 4pt]{0}{-4.15746}{x -0.414214 mul}

%{4.15746, 0.414214}, {3.18198, 1.}, {1.72208, 2.41421}, {0, \
%ComplexInfinity}, {-1.72208, -2.41421}, {-3.18198, -1.}, {-4.15746, \
%-0.414214}

\psline[linestyle=dashed,dash=3pt 2pt]{-}(-4,0.5)(4,0.5)
\psline[linestyle=dashed,dash=3pt 2pt]{-}(-4,1)(4,1)
\psline[linestyle=dashed,dash=3pt 2pt]{-}(-4,1.5)(4,1.5)
\psline[linestyle=dashed,dash=3pt 2pt]{-}(-4,2)(4,2)
\psline[linestyle=dashed,dash=3pt 2pt]{-}(-4,2.5)(4,2.5)
\psline[linestyle=dashed,dash=3pt 2pt]{-}(-4,3)(4,3)
\psline[linestyle=dashed,dash=3pt 2pt]{-}(-4,3.5)(4,3.5)
\psline[linestyle=dashed,dash=3pt 2pt]{-}(-4,4)(4,4)
%\psline[linestyle=dashed,dash=3pt 2pt]{-}(1,-2)(1,4)
\end{pspicture}
\caption{Orthogonal projection to the hyperbolic plane of the orbits relative to the three types of Killing vector fields: the horizontal lines are the orbits of $G_{24},\,a\neq 0$ (for $a=0$ the orbits are lines orthogonal to the hyperbolic plane); the lines through the origin are the orbits of $G_{34}$; the circles are the orbits of $G_{14}$.}
\label{fig:orbits}
\end{figure}

To give a geometric meaning of the Clairaut's relation in $\h^2\times\r$ we compute the function $\omega$ for the three types 
of invariant surfaces. To this aim we first recall the formula for the hyperbolic distance between two points in the half-plane model:
$$
d_{\h}(p,q)=\left\{
\begin{array}{lll}
\displaystyle{\left|\ln\left(\frac{x_p-\xi+R}{x_q-\xi+R}\;\frac{y_q}{y_p}\right)\right|}&\text{if}& x_p\neq x_q,\\
&&\\
\displaystyle{\left|\ln\left(\frac{y_q}{y_p}\right)\right|}&\text{if}& x_p= x_q,
\end{array}
\right.
$$
where $p=(x_p,y_p)$, $q=(x_q,y_q)$ while $R$ and $\xi$ represent, respectively,  the radius and the 
abscissa of the center of the geodesic through $p$ and $q$ (see Figure~\ref{fig:hyp-geodesic}).
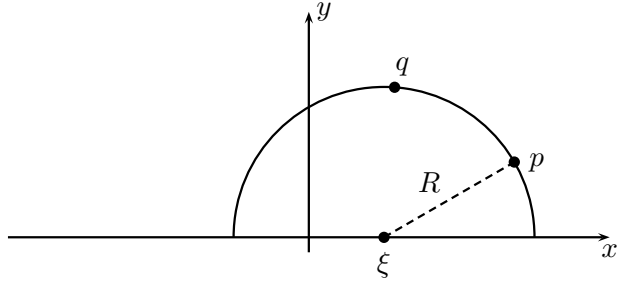
\begin{figure}[h!]
\begin{pspicture}(-4,-0.3)(4,3)
\psset{xunit=1cm,yunit=1cm,linewidth=.03}
\psaxes[labels=none,ticks=none]{->}(0,0)(-4,-0.2)(4,3)

\rput[t](4,-0.1){$x$}

\rput[l](0.1,3){$y$}
\rput[t](1,-0.2){$\xi$}
\rput[l](2.93205,1){$p$}
\rput[b](1.24147, 2.15){$q$}
\rput[b](1.6,0.6){$R$}

\psline[linestyle=dashed,dash=3pt 2pt]{-}(1,0)(2.73205,1)
\psarc(1,0){2}{0}{180}
\psdots[dotsize=0.15](1,0)
\psdots[dotsize=0.15](2.73205,1)
\psdots[dotsize=0.15](1.14147, 1.99499)

\end{pspicture}
\caption{Hyperbolic geodesic through $p$ and $q$.}
\label{fig:hyp-geodesic}
\end{figure}

We have
%\begin{itemize}

{\bf $\mathbf G_{24}$-surfaces}. In this case $\omega^2=(a^2+b^2y^2)/y^2$. An orbit through $p_0=(x_0,y_o,z_0)$ is a line contained in the plane $y=y_0$. Thus the hyperbolic distance of any point of the orbit to the plane $y=1$ is constant and equal to $d=|\ln(y_0)|$. Then the  Clairaut's relation becomes:
$$
\sqrt{a^2\,e^{2\epsilon d}+b^2} \cos\theta=c,
$$
with $\epsilon=\pm 1$ according to the sign of $(1-y_0)$. When $a=0$ the surface is invariant by vertical translations. In this case $\omega$ is constant everywhere which means that a geodesic must cut all orbits by the same angle.

{\bf $\mathbf G_{34}$-surfaces}. Introducing cylindrical coordinates $(r,\alpha,z)$ in $\h^2\times\r$, where $(r,\alpha)$ are polar coordinates in $\h^2$, a straightforward  computation gives $\omega^2=(1 + b^2\sin^2\alpha)/\sin^2\alpha$.
As an orbit belongs to a vertical plane through the $z$-axes, all of its points have constant hyperbolic distance from the plane $x=0$ equal to 
$$
d=\ln\left(\frac{\sqrt{1+4\tan^2\alpha}+1}{\sqrt{1+4\tan^2\alpha}-1}\right).
$$
Then, computing $\omega$ in terms of $d$, yields  the Clairaut's relation 
$$
\sqrt{2 \cosh d+b^2-1} \cos\theta=c.
$$

{\bf $\mathbf G_{14}$-surfaces}. This is the most interesting case, in fact the orbits are helices whose projections into the hyperbolic plane are geodesic circles with center at the point $i=(0,1,0)$. In fact, the hyperbolic distance from any point $p=(x,y,0)$ of the projection of the orbit of a fixed point $p_0=(x_0,y_0,z_0)$,  to $i$ is constant and equal to 
$$
d=\ln\left(\frac{\beta+\sqrt{\beta^2-4}}{2}\right),\quad \beta=\frac{1+x_0^2+y_0^2}{y_0}.
$$
A direct check shows that
$$
\omega^2=\frac{\beta^2}{4}+b^2-1.
$$
We then get the Clairaut's relation 
\begin{equation}\label{eq:clairaut-g124*}
\sqrt{\sinh^2d+b^2} \cos\theta=c.
\end{equation}
When $b=0$ the orbits of $G_1$ are geodesic circles and the invariant surfaces are called {\em rotational surfaces}. 
\subsection{The Clairaut's relation for rotational surfaces in the Bianchi-Cartan-Vranceanu spaces}
The Bianchi-Cartan-Vranceanu spaces (see \cite{Bi,Ca,Vr}) are described by  the following two-parameter 
family of Riemannian metrics 
\begin{equation}\label{1.1}
 g_{\ell,m} =\frac{dx^{2} + dy^{2}}{[1 + m(x^{2} + y^{2})]^{2}} +  \left(dz + 
\frac{\ell}{2} \frac{ydx - xdy}{[1 + m(x^{2} + y^{2})]}\right)^{2},\quad \ell,m \in {\r}
\end{equation}
defined on $M=\r^3$ if $m\geq 0$ and  on $M=\{(x,y,z)\in\r^3: x^2+y^2<-1/m\}$ otherwise.
Their geometric interest lies in the 
following fact: {\it the family of metrics (\ref{1.1}) 
includes all three-dimensional 
homogeneous metrics whose group of isometries has dimension $4$ or  
$6$, except for those of constant negative sectional curvature}.
The group of isometries of these spaces contains a one-parameter subgroup isomorphic to $SO(2)$.
The surfaces invariant by the action of $SO(2)$ are clearly called {\em rotational surfaces}. 
If we assume that the symmetry axes is the $z$-axes, then the infinitesimal generator of the group 
$SO(2)$ is the Killing vector field
$$
X=y\frac{\partial}{\partial x}-x\frac{\partial}{\partial y}.
$$
The orbits of $X$ are geodesic circles on horizontal planes with centre on the $z$-axes and the Clairaut's relation becomes
\begin{equation}\label{eq:clairaut-h3}
\omega\, \cos\theta=\sqrt{g_{m,\ell}(X,X)}\, \cos\theta=\frac{r}{2(1+m\, r^2)}\sqrt{4 +\ell^2 r^2} \cos\theta = c,
\end{equation}
where $r$ represents the Euclidean radius of the orbit and $\theta$ is the angle between the velocity vector of the geodesic and $X$. 
This Clairaut's relation was first found by P.~Piu and M.~Profir in \cite{PiuProfir} by a direct computation.
We now write down the Clairaut's relation in terms of the geodesic radius $d_{m,\ell}$ of the orbit.  For this, we recall that the geodesic on $(M,g_{\ell,m})$ tangent at the origin to the vector ${\partial}/{\partial x}$ is parametrized, according to the value of $m$, by:
$$
\begin{array}{lll}
\displaystyle{\alpha(s)=\left(\frac{1}{\sqrt{m}} \tan(\sqrt{m} s),0,0\right)}&\text{if}&m>0,\\
&&\\
\displaystyle{\alpha(s)=\left(\frac{1}{\sqrt{-m}} \tanh(\sqrt{-m} s),0,0\right)}&\text{if}& m<0,\\
&&\\
\displaystyle{\alpha(s)=\left(s,0,0\right)}&\text{if}& m=0.
\end{array}
$$
Since the curve $\alpha$ is parametrized by arc length the geodesic radius is 
$d_{m,\ell}=s_1$, where $\alpha(s_1)=(r,0,0)$. Replacing the value of $d_{m,\ell}$ in
\eqref{eq:clairaut-h3} we find the following geometric Clairaut's relations:

\begin{equation}\label{eq:clairaut-glm}
\begin{array}{cll}
\displaystyle{\frac{\sin(2\sqrt{m}\,d_{m,\ell})\sqrt{4m+\ell^2\tan^2(\sqrt{m}\,d_{m,\ell})}}{4m}\; \cos\theta = c}&\text{if}&m>0,\\
&&\\
\displaystyle{\frac{ \sinh(2\sqrt{-m}\,d_{m,\ell})\sqrt{\ell^2\tanh^2(\sqrt{-m}\,d_{m,\ell})-4m}}{-4m}\;\cos\theta = c}&\text{if}& m<0,\\
&&\\
\displaystyle{\frac{d \sqrt{4+\ell^2 d^2}}{2}\;\cos\theta = c}&\text{if}& m=0.
\end{array}
\end{equation}

\begin{remark}
For $\ell=0$ and $m=-1/4$ the metric $g_{-{1}/{4},0}$ is isometric to the metric \eqref{eq:metric-h2xr}, this is in agreement with the fact that  the Clairaut's relation \eqref{eq:clairaut-glm}, for $\ell=0$ and $m=-1/4$, 
coincides with \eqref{eq:clairaut-g124*} for $b=0$.
\end{remark}

\section{Integral formula for the geodesics}

In this section we give an integral formula to parametrize, locally, the geodesics on an invariant surface which are not orbits.

\begin{lemma}\label{tgeod}
Let ${M}^2\subset ({N}^3,g)$ be a $G_X$-invariant surface locally parametrized by
$\psi(u,v)$ (see \eqref{eq-psi}) and let $\alpha (s)=\psi(u(s),v(s))$ be a geodesic parametrized by arc length,
which is not an orbit, and with slant $c$. Then the following holds:
 \begin{equation}\label{eqsist}
    \left\{\begin{aligned}
    &F(u(s))\,{u'}+\omega(u(s))^2\,{v'}=c,\\
    &{u'}^2=1-\frac{c^2}{\omega(u(s))^2}.
    \end{aligned}\right.
 \end{equation}
Conversely, if system~\eqref{eqsist} is satisfied and ${u'}\neq 0$, then $\alpha$ is a geodesic parametrized by arc length and slant $c$.
 \end{lemma}
 \begin{proof}
Firstly, observe that the first equation of \eqref{eqsist} coincides with the second of \eqref{eqgeod} and, also,
it implies that
\begin{equation}\label{eqc}
2\,F(u(s))\,{u'}\,{v'}+\omega(u(s))^2\,{v'}^2=\frac{c^2-F(u(s))^2\,{u'}^2}{\omega(u(s))^2}.
\end{equation}

Now, if $\alpha$ is a geodesic parametrized by arc length, then the vector field
$\alpha'(s)=\psi_u\,u'(s)+\psi_v\,v'(s)$ satisfies
\begin{equation}\label{arc}
   1=g(\alpha',\alpha')= E(u(s))\, u'^2+2\, F(u(s))\,u'\,v'+\omega (u(s))^2\,v'^2.
\end{equation}
Therefore, substituting  \eqref{eqc} in \eqref{arc}, and using \eqref{eq-E} (i.e. $E\,\omega^2-F^2=\omega^2$), we obtain the second equation of \eqref{eqsist}: 
\begin{equation}
\nonumber
%\begin{aligned}
1=\frac{E(u(s))\,\omega(u(s))^2-F(u(s))^2}{\omega(u(s))^2}\,{u'}^2+\frac{c^2}{\omega(u(s))^2}={u'}^2+\frac{c^2}{\omega(u(s))^2}.
%\end{aligned}
\end{equation}

Conversely, if system~\eqref{eqsist} is satisfied, from \eqref{eq-E}  and \eqref{eqc}, we have
\begin{equation}\label{eqarc2}\begin{aligned}
1&={u'}^2+\frac{c^2}{\omega(u(t))^2}=
\bigg(\frac{E(u(t))\,\omega(u(t))^2-F(u(t))^2}{\omega(u(t))^2}\bigg)\,{u'}^2+\frac{c^2}{\omega(u(t))^2}\\
&=E(u(t))\,{u'}^2+\frac{c^2-F(u(t))^2\,{u'}^2}{\omega(u(t))^2}\\&=E(u(t))\,{u'}^2+2\,F(u(t))\,{u'}\,{v'}+\omega(u(t))^2\,{v'}^2\\&=
g(\alpha',\alpha'),
\end{aligned}
\end{equation}
so that $\alpha$ has unit speed. Finally, since $\alpha$ is not an orbit, from Remark~\eqref{re:eqgeod-bis}, we conclude.
%It remains to show that the first geodesic equation holds. In fact, differentiating with respect to $t$
% equation~\eqref{eqarc2} and using that
%$\alpha$ is not an orbit (i.e. $u'(t)\neq 0$), we have
%\begin{equation}\label{arc1}
%    E\,u''+F\, v''+\frac{E_u\,u'^2}{2}=-\frac{( F\,v'\,u''+F_u\,u'^2\,v'+\omega^2\,v'\,v'')}{u'}-\omega\,\omega_u\,v'^2.
%\end{equation}
%Thus,
%\begin{equation}\nonumber
%\begin{aligned}
%E\,u''+F\, v''+\frac{E_u\,u'^2}{2}-\omega\,\omega_u\,{v'}^2&=-\frac{{v'}}{{u'}}\,(F\,u''+F_u\,u'^2+2\,\omega \,\omega_u\,u'\,v'+\omega^2\, v'')\\
%&=-\frac{{v'}}{{u'}}\,(F\,u'+\omega^2\,v')'\\&=0.
%\end{aligned}
%\end{equation}
 \end{proof}

Integrating system~\eqref{eqsist} we have the following 
\begin{theorem}
Every geodesic on a $G_X$-invariant surface ${M}^2\subset ({N}^3,g)$, which is not an orbit, can be locally parametrized by $\beta(u)=\psi(u,v(u))$, where
 \begin{equation}\label{eq-integral}
    v(u)=\int_{u_0}^u \left(\frac{-F}{\omega^2} \pm \frac{c}{\omega\sqrt{\omega^2-c^2}}\right)\,du.
\end{equation}
and $c$ is the slant of $\alpha$.
\end{theorem}
\begin{proof}
Suppose that $M$ is locally parametrized by
$\psi(u,v)$ (see \eqref{eq-psi}) and let $\alpha (s)=\psi(u(s),v(s))$ be a
geodesic on $M$ parametrized by arc length, that is not an orbit. As ${u'}\neq 0$ we can, locally, invert the function $u=u(s)$ obtaining $s=s(u)$ and, therefore, we can consider the parametrization of $\alpha$ given by
$$\beta(u)=\alpha(s(u))=\psi(u,v(u)),\qquad v(u)=v(s(u)).$$

Multiplying  the equation 
$$
E(u(s))\, u'(s)^2+2\, F(u(s))\,u'(s)\,v'(s)+\omega (u(s))^2\,v'(s)^2=g(\alpha',\alpha')=1
$$ 
by $(ds/du)^2$ we get
\begin{equation}\label{eqpri}
E+2\,F\,\frac{dv}{du}+\omega^2\,\Big(\frac{dv}{du}\Big)^2=\Big(\frac{ds}{du}\Big)^2.
\end{equation}
Also, from the second equation of \eqref{eqsist}, we have that 
\begin{equation}\label{eqpri1}
\frac{ds}{du}_{\big|u(s)}=\frac{1}{{u'}(s)}=\frac{\omega(u(s))^2}{\omega(u(s))^2-c^2},\qquad c\in\r.
\end{equation}
Substituting \eqref{eqpri1} in \eqref{eqpri} we obtain
$$\omega^2\,\Big(\frac{dv}{du}\Big)^2+2\,F\,\frac{dv}{du}+E-\frac{\omega^2}{\omega^2-c^2}=0.$$
Now, using \eqref{eq-E}, we get
$$
F^2-\omega^2\,\Big[E-\frac{\omega^2}{\omega^2-c^2}\Big]=\frac{c^2\,\omega^2}{\omega^2-c^2}.
$$
Finally it results that
\begin{equation}\label{eqcor}
   \frac{dv}{du}=\frac{-F\pm\dfrac{c\,\omega}{\sqrt{\omega^2-c^2}}}{\omega^2},
\end{equation}
which implies that the equation of a geodesic segment (that is not an orbit) on an invariant surface is given by \eqref{eq-integral} as required.

\end{proof}

We now describe explicitly how to parametrize the invariant surfaces and, using \eqref{eq-integral}, how to plot the geodesics.

\begin{example}[The funnel surface]\label{ex:funnel}
Let consider the case of $G_{34}$ invariant surfaces in $\h^2\times\r$.  We shall use cylindrical 
coordinates $(r,\theta,z)$ for  $\h^2\times\r$ and the coordinates $\{\xi_1,\xi_2\}$ for the orbit space 
$\h^2\times\r/G_{34}=\{(\xi_1,\xi_2)\in\r^2\,:\, \xi_1\in(0,\pi)\}$, where $\{\xi_1,\xi_2\}$ are invariant functions with respect to the action of $G_{34}$.
Then, endowing the orbit space with the quotient metric 
$$
\tilde{g}=\frac{d\xi_1^2}{\sin^2\xi_1}+\frac{d\xi_2^2}{b^2+\sin^2\xi_1},
$$
the projection
% $\pi:\h^2\times\r\to \h^2\times\r/G_{34}$ defined by
$$
(r,\theta,z)\xmapsto{\pi}(\theta,z-b\ln r)
$$
becomes a Riemannian submersion. 
The simplest curve in $\h^2\times\r/G_{34}$, choosing $b=1$, is $\xi_2=0$, of which a parametrization by arc length is
$$
\tilde{\gamma}(u)=(2 \arccot e^{-u},0).
$$
A lift of $\tilde{\gamma}$ with respect to $\pi$ is
$$
\gamma(u)=(1,2 \arccot e^{-u},0).
$$
The corresponding invariant surface is parametrized, in rectangular coordinates and  according to \eqref{eq-psi} and \eqref{eq:orbitsg34}, by
\begin{eqnarray*}
\psi(u,v)&=&(e^v \cos(2 \arccot e^{-u}), e^v \sin(2 \arccot e^{-u}),v)\\
&=&(-e^v \tanh u, e^v \sech u,v).
\end{eqnarray*}

This surface is very well known because it is a complete minimal surfaces in $\h^2\times\r$ that can be thought as the graph of the function $z=\ln(x^2+y^2)/2$ and due  to its shape is known as the {\em funnel surface}. The coefficients of the induced metric are $E=1$, $F=0$ and $G=\omega^2=2+\sinh^2u$. Now, using  \eqref{eq-integral}, the geodesics, which are not orbits, can be parametrized by
$$
\alpha(u)=\psi\left(u,\int_{u_0}^u \frac{c}{\sqrt{\sinh^2t+2}\sqrt{2-c^2+\sinh^2t}}\,dt\right).
$$
To understand which orbits are geodesics we can use Proposition~\ref{orbit} and find that an orbit $u=u_0$ is a geodesic if and only if
$$
\omega_u(u_0)=\frac{\sinh u_0\, \cosh u_0}{\sqrt{\sinh^2u_0+2}}=0,
$$
that is $u_0=0$ and the corresponding slant is $c=\sqrt{2}$.
In Figure~\ref{fig-funnel} we show the plot of five geodesics through the point $p=(0,1,0)$ for different values of the slant $c$.
Moreover, in this case, all the curves with slant $c=0$ are geodesics. 
\begin{figure}[h!]
{\includegraphics[width=0.885\textwidth]{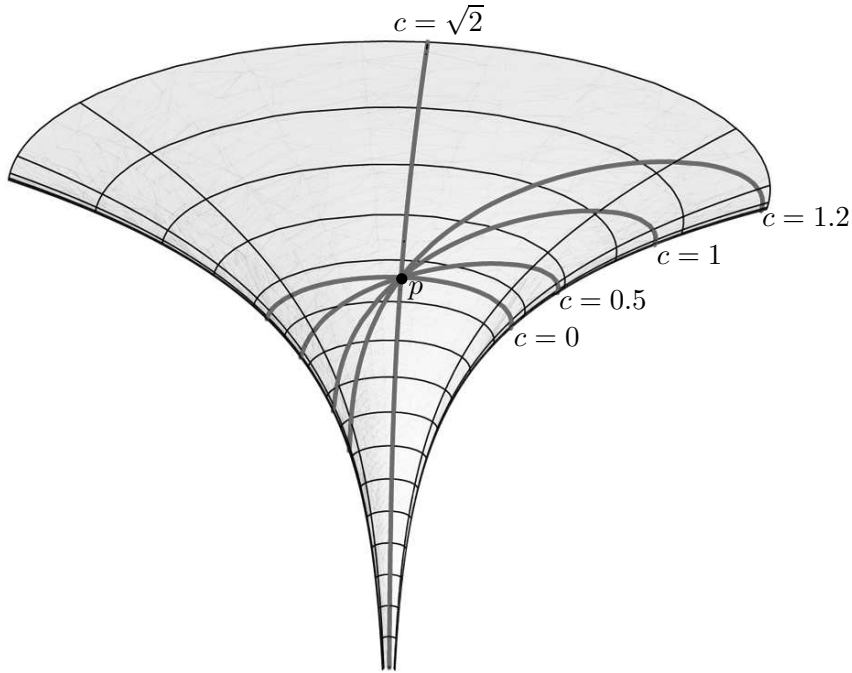}}
\begin{pspicture}(-5,0)(5,0)
\psset{xunit=1cm,yunit=1cm,linewidth=.03}
%\psaxes[labels=none,ticks=none]{->}(0,0)(-4,-0.2)(4,3)
\put(-.3,9.3){$c=\sqrt{2}$}
\put(4.6,6.7){$c=1.2$}
\put(3.2,6.2){$c=1$}
\put(1.9,5.6){$c=0.5$}
\put(1.3,5.1){$c=0$}
\put(-0.1,5.8){$p$}
\psdots[dotsize=0.15](-0.19,6.0)

\end{pspicture}

\caption{Five geodesics on the funnel surfaces though the point $p=(0,1,0)$ as seen from the viewpoint of coordinates $(1,10,-4)$; the geodesic with slant $c=\sqrt{2}$ is the only geodesic which is an orbit.}
\label{fig-funnel}

\end{figure}
\end{example}

\begin{remark}\label{re:lifthor}
Note that, in general, is rather difficult to parametrize an invariant surface using a horizontal lift of the profile curve. To illustrate this consider the case of $G_{34}$-invariant surfaces described in Example~\ref{ex:funnel}. Given a curve $\tilde{\gamma}(s)=(\xi_1(s),\xi_2(s))$ in the orbit space $\h^2\times\r/G_{34}=\{(\xi_1,\xi_2)\in\r^2\,:\, \xi_1\in(0,\pi)\}$, a horizontal lift is a curve $\gamma(s)=(r(s),\theta(s),z(s))$ such that
\begin{equation}\label{eq:lifthor}
\begin{cases}
\displaystyle{\theta=\xi_1+c_1},\quad c_1\in\r\\
\displaystyle{z=b\ln r +\xi_2+c_2},\quad c_2\in\r\\
\displaystyle{\frac{r'}{r}= -\frac{b \sin^2(\xi_1+c_1)}{b^2 \sin^2(\xi_1+c_1)+1}\xi_2'}.
\end{cases}
\end{equation}
The first two conditions of \eqref{eq:lifthor} guaranty that $\gamma$ is a lift, while the third 
one says that $\gamma'$ is orthogonal to the Killing vector field $X=X_3+bX_4$, i.e. $\gamma'$ is horizontal.  The difficulty in solving  \eqref{eq:lifthor} lies in the expression of the profile curve. In the case of the funnel surface $\xi_2=0$ and the solution is trivial. 

On the other hand, a lift, not necessarily horizontal,  of  $\tilde{\gamma}(s)=(\xi_1(s),\xi_2(s))$ must only satisfy
$$
\begin{cases}
\displaystyle{\theta=\xi_1+c_1}\\
\displaystyle{z=b\ln r +\xi_2+c_2}
\end{cases}
$$
and, for example, for the choice $r(s)=1$ we get that the curve
$\gamma(s)=(1,\xi_1(s)+c_1,\xi_2(s)+c_2)$ is a lift of any given profile curve.
\end{remark}
\section{Geodesics of invariant surfaces with constant Gauss curvature}

In this section we consider the case of a $G_X$-invariant surface ${M}^2\subset ({N}^3,g)$ such that the induced metric is of constant Gauss curvature. For this case we shall limit our investigation to the case when 
the lift $\gamma$, used to construct the parametrization of the surface \eqref{eq-psi}, is horizontal.  With this
assumption \eqref{eqcor} can be integrated on the same pattern as the case of rotational surfaces in the Euclidean space (see, for example, \cite[Pag. 185]{pressley}).

\begin{proposition}[Positive curvature] Let ${M}^2\subset ({N}^3,g)$ be a $G_X$-invariant surface of
constant positive Gauss curvature $K=1/R^2$, locally parametrized by
$\psi(u,v)$ (see \eqref{eq-psi}) with $\gamma$ horizontal lift. Then a geodesic on ${M}^2$, which is not an orbit, with slant $c\neq 0$, can be parametrized by
\begin{equation}
    v(u)=\frac{R}{\sqrt{a}}\,\arcsin\Big(-\frac{c\,R}{\sqrt{a-c^2}}\,\frac{{\omega}_u(u)}{\omega(u)}\Big)+b,\quad a,b\in\r,\, a> 0.
\end{equation}
\end{proposition}
\begin{proof}
First, as $K=1/R^2$, from \eqref{eq-main}, we have
\begin{equation}\label{erre}
     R^2\omega_{uu} (u)+\omega (u)=0.
\end{equation}
From {\eqref{erre}} it results that
$$
\frac{d}{du}(\omega(u)^2+R^2\,\omega_u(u)^2)=2\,\omega_u(u)\,(\omega(u)+R^2\,\omega_{uu}(u))=0,
$$
which implies that there exists a constant $a\in\r,\, a> 0,$ such that
\begin{equation}\label{bi}
\omega(u)^2+R^2\,\omega_u(u)^2=a.
\end{equation}
Combining \eqref{erre} and \eqref{bi}, we find
\begin{equation}\label{bi1}
   {\omega}_{uu}\,\omega-{\omega}_{u}^2=-\frac{a}{R^2}.
\end{equation}
Also, from \eqref{bi}, and taking into account Remark~\ref{re:omegac}, we have
 $$
 0<\omega^2-c^2=(a-c^2)-R^2\,{\omega}_{u}^2,
 $$
 which implies that $(a-c^2)>0$.
We can then consider the change of variables 
 $$
 \eta(u)=-\frac{c\,R}{\sqrt{a-c^2}}\,\frac{{\omega}_{u}(u)}{\omega(u)}.
 $$
Therefore, taking into account \eqref{bi1}, we get
\begin{equation}\label{dw}
    d\eta=\frac{c\,a}{R\,\sqrt{a-c^2}}\,\frac{du}{\omega^2}
\end{equation}
and, using \eqref{bi},
\begin{equation}\label{dww}
    \sqrt{1-\eta^2}=\frac{\sqrt{a}\,\sqrt{\omega^2-c^2}}{\omega\,\sqrt{a-c^2}}.
\end{equation}
Finally, integrating  \eqref{eq-integral}, we have
\begin{equation}\label{eqvu}\begin{aligned}
    v(u)&=\int \dfrac{c}{\omega\,\sqrt{\omega^2-c^2}}\,du=
    \frac{R}{\sqrt{a}}\int\frac{d\eta}{\sqrt{1-\eta^2}}\\&=\frac{R}{\sqrt{a}}\,\arcsin \eta+b\\
    &=\frac{R}{\sqrt{a}}\,\arcsin\Big(-\frac{c\,R}{\sqrt{a-c^2}}\,\frac{{\omega}_{u}(u)}{\omega(u)}\Big)+b,\quad b\in\r.
\end{aligned}
\end{equation}
\end{proof}

Let now consider the case of constant negative curvature. Before doing this note that, as $K=-1/R^2$,  \eqref{eq-main} becomes
\begin{equation*}\label{erre1}
     \omega(u)-R^2\omega_{uu} (u)=0,
\end{equation*}
which implies that
\begin{equation*}\label{bi2}
\omega(u)^2-R^2\,\omega_u(u)^2=a,\quad a\in\r.
\end{equation*}
And the latter two imply 
\begin{equation}\label{bi3}
   {\omega}_{uu}\,\omega-{\omega}_u^2=\frac{a}{R^2}.
\end{equation}
In this case, differently from the case of positive curvature, the constant $a$ can be any real number.
Performing changes of variables, similar to the case of constant positive curvature, \eqref{eq-integral}
can be integrated and gives:
\begin{proposition}[Negative curvature] Let ${M}^2\subset ({N}^3,g)$ be a $G_X$-invariant surface of
constant negative Gauss curvature $K=-1/R^2$, locally parametrized by
$\psi(u,v)$ (see \eqref{eq-psi}) with $\gamma$ horizontal lift. Then a geodesic on ${M}^2$, which is not an orbit, with slant $c\neq 0$, can be parametrized by
\begin{equation*}
  \begin{array}{lll}
   \displaystyle{v(u)=\frac{\sqrt{\omega^2-c^2}}{c\,{\omega}_u}+b}, & \text{if} & a=0,\\
   &&\\
   \displaystyle{ v(u)=\frac{R}{\sqrt{-a}}\,\arcsin\Big(-\frac{c\,R}{\sqrt{c^2-a}}\,\frac{\omega_u(u)}{\omega(u)}\Big)+b}, & \text{if} & a<0,\\
    &&\\
   \displaystyle{ v(u)=\frac{R}{\sqrt{a}}\,\ln\Big(\frac{c\,R\,{\omega_u}+\sqrt{a\,(\omega^2-c^2)}}{\omega\,\sqrt{c^2-a}}\Big)+b},& \text{if} & 0<a<c^2,\\
    &&\\
    \displaystyle{ v(u)=\frac{R}{\sqrt{a}}\,\sinh^{-1}\Big(\frac{c\,R}{\sqrt{a-c^2}}\,\frac{{\omega_u}(u)}{\omega(u)}\Big)+b},& \text{if} & a>c^2,\\
     &&\\
    \displaystyle{ v(u)=\frac{R}{2c}\ln\Big(\frac{\omega_u^2}{\omega^2}\Big)+b}, & \text{if} & a=c^2,
  \end{array}
\end{equation*}
where $b\in\r$.
\end{proposition}

In the last case, that is when the Gauss curvature is zero, we have that $\omega_u=a\in\r$, and it can be handled  in the same way as before, giving
 \begin{proposition}[Flat case] Let ${M}^2\subset ({N}^3,g)$ be a flat $G_X$-invariant surface, locally parametrized by
$\psi(u,v)$ (see \eqref{eq-psi}) with $\gamma$ horizontal lift. Then a geodesic on ${M}^2$, which is not an orbit, with slant $c\neq 0$, can be parametrized by

\begin{equation*}
  \begin{array}{lll}
   \displaystyle{  v(u)=\frac{1}{a}\,\arctan\left(\frac{c}{\sqrt{\omega^2-c^2}}\right)+b}, & \text{if} & a\neq 0,\\
   &&\\
   \displaystyle{ v(u)=\frac{c}{\omega\sqrt{\omega^2-c^2}}\,u+b}, & \text{if} & a=0,  \end{array}
\end{equation*}
where $b\in\r$.
\end{proposition}

\begin{remark}
The local expressions of the geodesics given in this section are particularly explicit. Nevertheless,
for being completely honest, we have to point out that they are true only in the case the invariant surface is parametrized by a horizontal lift of the profile curve and this, in the general case, makes things more complicated as explained in Remark~\ref{re:lifthor}.
\end{remark}


\begin{thebibliography}{10}



%\bibitem{ABRO} U.~Abresch and H.~Rosenberg. A Hopf differential for constant mean curvature surfaces in $\Bbb{H}^2 \times \Bbb{R}$. {\em Acta Math.} 193 (2004), 141--174.

%\bibitem{DA} B.~Daniel. Isometric immersions into 3-dimensional homogeneous manifolds.  {\em Comment. Math. Helv.}  82  (2007), 87--131.

\bibitem{Bi} L.~Bianchi. {\it Gruppi continui e finiti}. Ed. Zanichelli,
Bologna, 1928.
\bibitem{Ca} \'{E}.~Cartan. {\em Le\c{c}ons sur la g\'{e}om\'{e}trie
des espaces de Riemann}. Gauthier Villars, Paris, 1946.
\bibitem{RCPPAR1} R.~Caddeo, P.~Piu and A.~Ratto. ${SO}(2)$-invariant minimal and constant mean curvature surfaces in three dimensional homogeneous spaces. {\em Manuscripta Math.} 87 (1995), 1--12.
\bibitem{RCPPAR2} R.~Caddeo, P.~Piu and A.~Ratto. Rotational surfaces in $\Bbb{H}_{3}$ with constant Gauss curvature. {\em  Boll. Un. Mat. Ital. B}  10 (1996), 341--357.

%\bibitem{FEMI} I.~Fernandez and P.~Mira. Harmonic maps and constant mean curvature surfaces in $\Bbb{H}^2 \times \Bbb{R}$. {\em Amer. J. Math.} 129 (2007), 1145--1181.

\bibitem{FIMEPE} C.B.~Figueroa, F.~Mercuri and R.H.L.~Pedrosa. Invariant surfaces of the Heisenberg groups. {\em Ann. Mat. Pura Appl.} 177 (1999), 173--194.

\bibitem{LO} R.~Lopez. Invariant surfaces in homogenous space Sol with constant curvature. {\tt arXiv:0909.2550}.

\bibitem{MO1}
S. Montaldo, I.I. Onnis.
\newblock Invariant CMC surfaces in $\Bbb H^2\times\Bbb R$.
\newblock {\em  Glasg. Math. J.}  46  (2004), 311--321. 
\bibitem{MO}
S. Montaldo, I.I. Onnis.
\newblock Invariant surfaces in a three-manifold with constant
Gaussian curvature.
\newblock {\em J. Geom. Phys.} 55  (2005), 440--449.
%
\bibitem{Olver}
P.J.~Olver.
\newblock {\em Application of Lie Groups to Differential Equations}.
\newblock  GTM 107, Springer-Verlag, New York, 1986.
%
\bibitem{Onnis} I.I.~Onnis.
Invariant surfaces with constant mean curvature in $\h^2\times\r$.
{\em Ann. Mat. Pura Appl.} 187 (2008), 667--682.


\bibitem{pressley}  A.~Pressley. {\em Elementary Differential Geometry}.
Springer Undergraduate Mathematics Series. Springer-Verlag, 2001.


\bibitem{PiuProfir} P.~Piu and M.~Profir.
On the geodesic of the rotational surfaces in the Bianchi-Cartan-Vranceanu spaces.
{\em VIII International Colloquium on Differential Geometry}, Ed. J.A.~Alvarez Lopez and E.~Garcia Rio  (2009), 306--310.

%\bibitem{Rosenberg}
%H. Rosenberg.
%\newblock Home page.
%\newblock {\tt http://people.math.jussieu.fr/\textasciitilde rosen/}.

\bibitem{SETO} R.~S\'a Earp and E.~Toubiana. Screw motion surfaces in $\Bbb{H}^{2}\times \Bbb{R}$ and $\Bbb{S}^{2}\times \Bbb{R}$. {\em Illinois J. Math. } 49 (2005), 1323--1362.

\bibitem{SE} R.~S\'a Earp. Parabolic and hyperbolic screw motion surfaces in $\Bbb{H}^{2}\times \Bbb{R}$.  {\em J. Aust. Math. Soc.}  85  (2008), 113--143.

\bibitem{TO} P.~Tomter. Constant mean curvature surfaces in the Heisenberg group.   
{\em Proc. Sympos. Pure Math.} 54, 485--495, Amer. Math. Soc., Providence, RI, 1993. 

\bibitem{Vr}  G.~Vranceanu. {\em Le\c{c}ons de g\'{e}om\'{e}trie
diff\'{e}rentielle}.
Ed. Acad. Rep. Pop. Roum., vol I, Bucarest, 1957.

\end{thebibliography}
\end{document}